\documentclass[final,leqno,onefignum,onetabnum]{siamltex1213}
\usepackage{algorithm}
\usepackage{siunitx}
\usepackage{amsmath,amssymb}
\usepackage{algorithmic}
\usepackage{color}
\newcommand{\rev}[1]{{#1}}

\title{An optimization approach for well-targeted transcranial direct current stimulation} 

\author{Sven Wagner \footnotemark[2] \and Martin Burger \footnotemark[3] \and Carsten H. Wolters \footnotemark[2]}

\begin{document}
\maketitle
\slugger{siap}{xxxx}{xx}{x}{x--x}%slugger should be set to mms, siap, sicomp, sicon, sidma, sima, simax, sinum, siopt, sisc, or sirev

\renewcommand{\thefootnote}{\fnsymbol{footnote}}

\footnotetext[2]{Institute for Biomagnetism and Biosignalanalysis, Westf\"alische Wilhelms-Universit\"at M\"unster, Malmedyweg 15, D 48149 M\"unster}
\footnotetext[3]{Institute for Computational and Applied Mathematics, Westf\"alische Wilhelms-Universit\"at M\"unster, Einsteinstr. 62, D 48149 M\"unster}

\renewcommand{\thefootnote}{\arabic{footnote}}

\begin{abstract}
Transcranial direct current stimulation is a non-invasive brain stimulation technique which modifies neural excitability by providing weak currents through scalp electrodes. The aim 
of this study is to introduce and analyze a novel optimization method for safe and well-targeted multi-array tDCS. For optimization, we consider an optimal control problem for a Laplace equation with Neumann boundary conditions with control and point-wise gradient state constraints. 
%We prove existence and residual and objective convergence results 
\rev{We prove well-posedness results 
for the proposed methods} and provide computer simulation results in a highly realistic six-compartment geometry-adapted hexahedral head model. For discretization of the proposed minimization problem the finite element method is employed and the existence of at least one minimizer to the discretized optimization problem is shown. For numerical solution of the corresponding discretized problem we employ the alternating direction method of multipliers and comprehensively examine the cortical current flow field with regard to focality, target intensity and orientation. The numerical results reveal that the optimized current flow fields show significantly higher focality and, in most cases, higher directional agreement to the target vector in comparison to standard bipolar electrode montages.
\end{abstract}


\begin{AMS}\end{AMS}

\pagestyle{myheadings}
\thispagestyle{plain}
\markboth{Sven Wagner, Martin Burger and Carsten H. Wolters}{An optimization approach for well-targeted transcranial direct current stimulation}

\section{Introduction}
Transcranial direct current stimulation (tDCS) is a non-in\-vasive, inexpensive and easy-to-perform brain stimulation technique which modifies neural excitability \cite{Nitsche2000}.
Changes in neural membrane potentials are induced in a polarity-dependent manner, for example, in a motor cortex study, anodal stimulation right over the motor cortex enhances cortical excitability, whereas cathodal stimulation inhibits it \cite{Nitsche2000}. Recently, tDCS has been applied successfully in the treatment of neurological and neuropsychiatric disorders such as epilepsy \cite{Fregni2006}, depression \cite{Boggio2007} and Alzheimer disease \cite{Ferrucci2008}. The effects of tDCS can be preserved for more than one hour after stimulation \cite{Nitsche2003}. \newline
The conventional strategy is to apply the current density $\sigma \nabla \Phi$ via two large electrodes, with the active electrode (anode) to be placed above the presumed 
target region and the reference electrode (cathode) far away from the target region \cite{Nitsche2000, Nitsche2003}. Accurate and detailed finite element (FE) head models have been created to investigate the induced current density distribution \cite{Wagner2014, Dmochowski2011, Sadleir2012}. While significant effects of stimulation as compared to sham were reported \cite{Boggio2007, Fregni2006, Nitsche2000}, computer simulation studies have revealed that the induced cortical current flow fields are very widespread with often strongest current density amplitudes in non-target brain regions \cite{Wagner2014, Sadleir2012}. It is therefore a matter of debate whether the effects of stimulation are driven by the target brain region or elicited by adjacent cortical lobes \cite{Sadleir2012}. \newline
In order to overcome the limitations of conventional bipolar electrode montages, algorithmic-based sensor optimization approaches were presented \cite{Dmochowski2011, Sadleir2012, Im2008, Ruffini2014}. Im and colleagues \cite{Im2008} searched for two electrode locations which generate maximal current flow towards a certain target direction. \rev{Ruffini and colleagues \cite{Ruffini2014} described a method for optimizing the configuration of multifocal tCS for stimulation of brain networks, represented by spatially extended cortical targets.} 
Dmochowski and colleagues \cite{Dmochowski2011} used a multi-channel array consisting of 64 fixed electrode locations to calculate optimized stimulation protocols for presumed target regions. They employed radial and tangential targets and reported that compared with conventional electrode montages their optimization approach achieved electric fields which exhibit simultaneously greater focality and target intensity at cortical targets using the same total current applied. 
\newline
While the optimization results are very promising and a pilot study using multi-array tDCS devices for rehabilitation after stroke has been conducted \cite{Dmochowski2013}, to our knowledge, there is currently no study providing an in-depth analysis of mathematical models for multi-array tDCS optimization.\newline 
Given a volume conductor model $\Omega$ with a fixed electrode arrangement and a target vector $\boldsymbol e$ in $\Omega_t$ with $\Omega_t \subset \Omega$ being the target cortical area, the optimization approach estimates an optimal applied current pattern at the fixed electrodes. A Laplace equation with inhomogeneous Neumann boundary conditions is used to calculate the induced current density distribution \cite{Wagner2014} which is to be controlled by the boundary condition to ensure safe and focused stimulation. Therefore, the optimization problem for tDCS is in the class of control problems with Neumann boundary conditions \cite{LionsBook}. \newline
\rev{In this paper, we modeled the electrodes with the point electrode model (PEM) 
%in combination with additional surface finite elements for the sponges, the so-called gap model 
\cite{Pursiainen2012, Agsten2015, Agsten2016}. They can, however, also be modeled using a complete electrode model (CEM) \cite{Pursiainen2012, Dannhauer2012,  Eichelbaum2014, Agsten2015, Agsten2016}. In \cite{Agsten2015, Agsten2016}, it has been shown, that CEM and PEM only lead to small differences which are mainly situated locally around the electrodes and are very small in the brain region. Based on these results, the application of PEM 
%and especially of the gap model like in the current study, being even closer to the CEM, 
is expected to result in negligible differences to the CEM and should thus provide a sufficiently accurate modeling of the current density within the brain region. } \newline
This paper is organized as follows: In Section 2 we establish the mathematical model of tDCS and introduce the optimization problem for multi-channel tDCS. In Section 3 we show existence of at least one minimizer to the optimization problem and the FE method is used for numerical discretization. In Section 4 we derive the algorithm for sensor optimization. Section 5 provides optimization results in a highly-realistic six-compartment head model (skin, skull compacta, skull spongiosa, CSF, gray and white matter) with white matter anisotropy. Furthermore, a conclusion and outlook section is presented. 
%Finally, in the Appendix, we prove residual and objective convergence results for the proposed optimization method.

\section{Mathematical Model of tDCS}
In order to calculate the current flow field induced by tDCS, the quasistatic approximation to Maxwell's equations is applied \cite{Wagner2014}. This yields the \textit{tDCS forward problem}
\begin{eqnarray*}
\nabla \cdot \sigma \nabla \Phi & = & 0 \qquad in \; \Omega\\
 \langle \sigma \nabla \Phi, \boldsymbol n \rangle & = & \boldsymbol I \qquad on \; \Gamma \subset \partial \Omega \\
 \Phi & = & 0 \qquad on \; \Gamma_D = \partial \Omega \setminus \Gamma
\end{eqnarray*}
with $\Phi$ being the electric potential, $\sigma \in (L^{\infty})^{3 \times 3}$ being an anisotropic conductivity tensor, $\boldsymbol I$ being the applied current pattern at the electrodes with non-zero values only at the electrode surfaces, $\boldsymbol n$ being the outward normal vector and $\Gamma$ being a part of the boundary of the domain $\Omega$. Furthermore, a Dirichlet boundary condition on the remaining part $\Gamma_D$ is used to ensure that a solution to the tDCS forward problem is unique. We assume all parts of the boundary to have nonzero measure and to be of Lipschitz regularity. We introduce the Sobolev space 
$$H_{\diamond}^{-\frac{1}{2}}(\Gamma) := \lbrace u \in H^{-\frac{1}{2}}(\Gamma) | \int_{\Gamma}\! u(s) \, \mathrm{d} \boldsymbol s = 0 \rbrace \subset H^{-\frac{1}{2}}(\Gamma)$$ with $H^{-\frac{1}{2}}(\Gamma)$ being the standard Sobolev space for Neumann boundary values on $\Gamma$. The precise definition of $H^{-\frac{1}{2}}(\Gamma)$ is to be the dual space of $H^{\frac{1}{2}}(\Gamma)$ consisting of the Dirichlet traces of $H^1$-functions in $\Omega$. The integral in the above definition has to be interpreted as a duality product with the constant function equal to one, which is in $H^{\frac{1}2}(\Gamma)$. The weak formulation of the forward problem is given by
$$
\int_\Omega \sigma \nabla \Phi \cdot \nabla \Psi~dx = \int_\Gamma {\boldsymbol I}~\Psi ~d\sigma,  
$$
for all test functions $\Psi \in H^1(\Omega)$ with vanishing Dirichlet trace on $\Gamma_D$. Note that the boundary integral on the right-hand side is only defined if ${\boldsymbol I} \in L^2(\Gamma)$, in general it is to be replaced by the duality product between  $H^{-\frac{1}{2}}(\Gamma)$ and $H^{\frac{1}{2}}(\Gamma). $
\newline
Under standard and naturally satisfied regularity assumptions, there exists a unique solution $\Phi \in H^1(\Omega)$ to the tDCS forward problem, which can be shown by the Lax-Milgram Lemma:
\begin{theorem}\label{thmwelldef}
Let $\sigma \in (L^{\infty})^{3 \times 3}$ with $\sigma \geq \sigma_0 I_{3 \times 3}$ and let $\boldsymbol I \in H_{\diamond}^{-\frac{1}{2}}(\Gamma)$. Then there exists a solution 
$\Phi \in H^1(\Omega)$ to the tDCS forward problem. The solution is unique if $\Gamma_D$ has positive measure. Moreover, the following norm estimate holds:
\begin{equation*}
 \left\| \Phi \right\|_{H^1} \leq C \left\| \boldsymbol I \right\|_{H^{-\frac{1}{2}}}
\end{equation*}
\end{theorem}
 \newline \noindent
For current density optimization, a multi-channel array consisting of 74 fixed electrode locations (the locations of an extended 10-10 EEG electrode system) is used and optimized applied current patterns are calculated for presumed targets. Figure \ref{Fig0} illustrates an overview of the optimization setup in a two-dimensional model. Optimally, the induced current density distribution is maximal in the target region $\Omega_t$ and zero in non-target regions $\Omega_r = \Omega \setminus \Omega_t$. Physically, a focal stimulation without stimulating non-target regions in not possible as the current has to flow through the volume to reach the target region. Therefore, the current density is restricted by $\epsilon > 0$ such that 
$$ |\sigma \nabla \Phi| \leq \epsilon \qquad \mbox{in~} \Omega_r. $$
 Secondly, the total current applied to all electrodes is limited to 2 mA, a commonly used safety \rev{criterion} \cite{Dmochowski2011}.
\begin{figure}[t!]
\begin{center}
\includegraphics*[width = 0.6\textwidth,keepaspectratio]{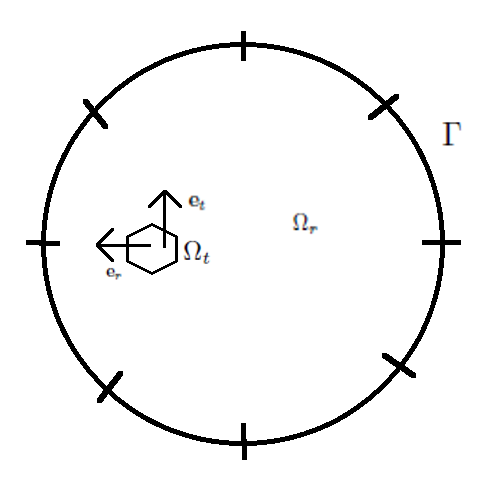}
\end{center}
\caption{Optimization setup in a two-dimensional model. The target region $\Omega_t$, a radial ($\boldsymbol e_r$) and tangential ($\boldsymbol e_t$) target vector and the boundary 
$\Gamma$ of the volume conductor model $\Omega$ are demonstrated. The electrodes are depicted with solid lines.}\label{Fig0}
\end{figure}\newline\noindent
For effective stimulation, the optimized current density distribution should be oriented perpendicularly to the cortical surface \cite{Bindman1962, Creutzfeldt1962}. Besides the correct target location, also the target direction is important as shown by Bindman and colleagues \cite{Bindman1962} and Creutzfeldt and colleagues \cite{Creutzfeldt1962} who were able to demonstrate in physical measurements in rats and cats, respectively, that the neural firing rate is strongly influenced by the direction of the field to the cortical surface. Perpendicularly-inwards (and parallel to the long apical dendrites of the large pyramidal cells in cortical layer V) stimulation, i.e. anodal stimulation, strongly enhanced the activity of the cortical neurons, whereas perpendicularly-outwards stimulation, i.e. cathodal stimulation, inhibited it. We thus maximize $\int_{\Omega_t}\! \langle \sigma \nabla \Phi, \boldsymbol e \rangle \, \mathrm{d} \boldsymbol x$ with $\boldsymbol e \subset \Omega_t$ being a \rev{unit} vector perpendicularly-inwards-oriented to the cortical surface \cite{Bindman1962, Creutzfeldt1962}. \newline
In the following we reformulate the current density constraint in the whole domain as
$\omega |\sigma \nabla \Phi| \leq \epsilon$ with a weight $\omega = 1$ in $\Omega_r$ and $0 < \omega \ll 1$ in $\Omega_t$. Thus, we consider the
constrained optimization problem 
\begin{eqnarray*}
  \mathbf{(P)} \qquad 
&- \int_{\Omega_t}\! \langle \sigma \nabla \Phi, \boldsymbol e \rangle \, \mathrm{d} \boldsymbol x& \; \rightarrow \min_{{\boldsymbol I} \in H^{-\frac{1}{2}}_{\diamond}(\Gamma)} \\ 
& \text{subject to } &
 \omega | \sigma \nabla \Phi | \leq \epsilon \\ 
 && \int_{\Gamma}\! |\boldsymbol I| \, \mathrm{d} \boldsymbol x \;  \leq 4 \\
 && \nabla \cdot \sigma \nabla \Phi  =  0 \qquad \quad in \; \Omega \\
 && \langle \sigma \nabla \Phi, \boldsymbol n \rangle  = \boldsymbol I \qquad \; on \;  \Gamma   \\
 && \Phi = 0 \qquad \qquad \qquad on \; \Gamma_D 
\end{eqnarray*}
which is a control problem with Neumann boundary conditions \cite{LionsBook}. Note that safety limitation on the inflow current is a \textit{control constraint}, 
while the limit on the current density outside the target region is a \textit{state constraint}. 
Also, the state constraint with $\omega > 0$ ensures that $\nabla \Phi$ is bounded in $L^{\infty}(\Omega)$ and due to the embedding theorem also bounded in $L^2(\Omega)$. This implies that $\Phi \in H^1(\Omega)$ and due to the trace theorem $\Phi|_\Gamma \in H^{\frac{1}{2}}(\Gamma)$. Since $\sigma \in (L^{\infty})^{3 \times 3}$ this implies that $ \langle \sigma \nabla \Phi, \boldsymbol n \rangle$ is bounded in $H^{-\frac{1}{2}}(\Gamma)$. \noindent
\begin{remark}
The constrained optimization problem $(\mathbf{P})$ is a rather challenging optimal control problem. Firstly, the current density distribution $\sigma \nabla \Phi$ is optimized, while many other applications require to optimize the potential field $\Phi$. Secondly, a point-wise gradient state constraint is used. While optimal control problems with state constraints were frequently used \cite{Hartl1995}, not much attention was given to gradient state constraints only until recently \cite{Schiela2011,wollner2,wollner1} focusing on distributed control in the domain however. Thirdly, we look for a control function $\boldsymbol I$ vanishing on large parts of the boundary, known as \textit{sparsity optimal control problem} \cite{Herzog2012}. Indeed, we rather look for $\boldsymbol I$ as a combination of concentrated measures rather than an $L^1$-function, so the above integral formulation has to be understood as formal notation for the total variation of the Radon measure identified with $\boldsymbol I$. This will be made precise in the next section.
\end{remark}\newline
Applying Gauss' Theorem to $\int_{\Omega_t}\! \langle \sigma \nabla \Phi, \boldsymbol e \rangle \, \mathrm{d} \boldsymbol x$ provides a further motivation for the proposed objective functional. Assuming $\sigma$ and $\boldsymbol e$ locally constant we have 
$$
 \int_{\Omega_t}\! \langle \sigma \nabla \Phi, \boldsymbol e \rangle \, \mathrm{d} \boldsymbol x +
 \int_{\Omega_t}\; \sigma \Phi \nabla \cdot \boldsymbol e + \boldsymbol e \nabla \sigma \Phi \, \mathrm{d} \boldsymbol x = 
 \int_{\Omega_t}\! \nabla \cdot (\sigma \Phi\boldsymbol e ) \, \mathrm{d} \boldsymbol x = 
 \int_{\partial \Omega_t}\! \langle \sigma \Phi \boldsymbol e , \boldsymbol n \rangle \mathrm{d}\sigma
$$
In order to maximize the current densities along the target direction, the potential should be maximized where the target vector $\boldsymbol e$ is oriented parallel to the outward normal vector $\boldsymbol n$ on $\partial \Omega_t$ and minimized where it is antiparallel.

\section{Optimal Control Problem}
The goal of this section is to provide insight into the existence theory of at least one minimizer to the tDCS optimization problem $\mathbf{(P)}$. Furthermore, we show 
existence of at least one minimizer to a simplified optimization problem without control constraint at the electrodes. Finally, the finite element method is used for numerical discretization of
the considered optimization procedure.

\subsection{Continuous formulation}
We now reconsider the optimization problem $\mathbf{(P)}$. In order to provide a rigorous formulation of the safety constraint we interpret $\boldsymbol I$ as a Radon measure in the space $\mathcal{M}(\Gamma)$, whose norm, i.e. the total variation of a measure, is an appropriate replacement for the $L^1$-norm.  The constraint then becomes
$$ \left\| \boldsymbol I \right\|_{\mathcal{M}(\Gamma)} \leq 4  $$
In order to obtain a unified formulation also allowing for other design constraints respectively a discrete electrode setup, we introduce a feasible set $\mathcal{D}(\Gamma) \subset  H_{\diamond}^{-\frac{1}{2}}(\Gamma)$.
To further simplify notation we define an operator $A$ via
$$ A: \quad  H_{\diamond}^{-\frac{1}{2}}(\Gamma) \rightarrow L^2(\Omega)^3, \quad    \boldsymbol I \mapsto \sigma \nabla \Phi. $$\noindent
Since $\boldsymbol I \in \mathcal{D}(\Gamma) \subset H_{\diamond}^{-\frac{1}{2}}(\Gamma)$, the boundary value problem for the Poisson equation has a unique solution $\Phi \in H^1(\Omega)$ bounded by a multiple of the norm of $\boldsymbol I$ as guaranteed by Theorem \ref{thmwelldef}. This implies that $A$ is a well-defined linear operator. Hence we can substitute $\sigma \nabla \Phi = A \boldsymbol I$ in $\mathbf{(P)}$, which leads to the following equivalent reformulation
\begin{eqnarray*}
  \mathbf{(P^2_{\delta})} \qquad 
&- \int_{\Omega_t}\! \langle A \boldsymbol I, \boldsymbol e \rangle \, \mathrm{d} \boldsymbol x& \rightarrow \min_{{\boldsymbol I} \in {\mathcal D}(\Gamma)} \\ 
& \text{subject to } &
 \omega | A \boldsymbol I |  \leq \delta \\ 
 && \left\| \boldsymbol I \right\|_{\mathcal{M}(\Gamma)} \;  \leq 4.
\end{eqnarray*}\noindent
The existence of at least one minimizer $\boldsymbol I \in \mathcal{D}(\Gamma)$ to $\mathbf{(P^2_{\delta})}$ is the topic of the following theorem.

\begin{lemma} \label{newlemma}
Let $\omega$ be bounded away from zero in $\Omega$ and let ${\boldsymbol I} \in  H_{\diamond}^{-\frac{1}{2}}(\Gamma)$ be such that $A {\boldsymbol I}$ satisfies the constraint
$\omega | A \boldsymbol I |  \leq \delta$ almost everywhere in $\Omega$. Then there exists a constant $C$ depending only on the given data $\sigma$, $\omega$, $\delta$, $\Omega$ and $\Gamma$, such that
\begin{equation}
	\Vert \boldsymbol I \Vert_{H^{-\frac{1}{2}}(\Gamma)} \leq C
\end{equation}
\end{lemma} 
\begin{proof}
We have 
$$ 	\Vert \boldsymbol I \Vert_{H^{-\frac{1}{2}}(\Gamma)} = \sup \{ \langle \boldsymbol I, \psi \rangle ~|~ \Vert \psi \Vert_{H^{\frac{1}{2}}(\Gamma)} \leq 1\}.$$
By the trace theorem we can write each such $\psi$ as the Dirichlet trace of some $\Psi \in H^1(\Omega)$ with vanishing trace on $\Gamma_D$ and $\Vert \Psi \Vert_{H^1(\Omega)} \leq C_T$ for a universal constant in the trace theorem depending only on $\Omega$ and $\Gamma$. Hence, by the weak formulation of the elliptic partial differential equation satisfied by $\Phi$ with $A \boldsymbol I = \sigma \nabla \Phi$, we find
$$ 	\Vert \boldsymbol I \Vert_{H^{-\frac{1}{2}}(\Gamma)} \leq \sup \{ \int_\Omega (A\boldsymbol I) \cdot \nabla \Psi ~dx ~|~ \Vert \Psi \Vert_{H^1(\Omega)} \leq C_T  \}.$$
Using the Cauchy-Schwarz inequality and an elementary estimate of the $L^2$-norm of $A \boldsymbol I$ in terms of its supremum norm, we find
$$ \Vert \boldsymbol I \Vert_{H^{-\frac{1}{2}}(\Gamma)} \leq C = C_T \sqrt{|\Omega|~\frac{\delta}{\inf \omega}}, $$
and we see that $C$ only depends on $\sigma$, $\omega$, $\delta$, $\Omega$ and $\Gamma$.
\end{proof}\noindent

\begin{theorem}\label{minQE}
Let ${\mathcal D}(\Gamma)$ be weakly closed in $H_{\diamond}^{-\frac{1}{2}}(\Gamma)$, let $0 \in {\mathcal D}(\Gamma)$ and let $\epsilon > 0$. Then there exists at least one minimizer $\boldsymbol I \in \mathcal{D}(\Gamma)$ of the convex optimization problem $\mathbf{(P^2_{\delta})}$. 
\end{theorem}
\begin{proof}
An easy computation reveals that $\boldsymbol I = 0 \in \mathcal{D}(\Gamma)$ is feasible for $\mathbf{(P^2_{\delta})}$, i.e. the set of feasible points is not empty. Moreover, it is closed in the weak-star topology of $H_{\diamond}^{-\frac{1}{2}}(\Gamma)\cap {\cal M}(\Gamma)$, which can be seen as follows: First of all the norm in ${\cal M}(\Gamma)$ is lower semicontinuous with respect to weak-star convergence in this space, which implies that a limit of feasible points satisfies the safety constraint as well. Moreover, weak convergence of a sequence $\boldsymbol  I^k$ in $H_{\diamond}^{-\frac{1}{2}}(\Gamma)$ implies weak convergence of $A\boldsymbol I^k$ in $L^2(\Omega)^3$, which implies that the limit satisfies the state constraint due to the weak closedness of pointwise constraints in $L^2$. Moreover, the feasible set is obviously bounded in ${\mathcal M}(\Omega)$ due to the safety constraint and in $H_{\diamond}^{-\frac{1}{2}}(\Gamma)$ due to Lemma \ref{newlemma}. Hence, the Banach-Alaoglu Theorem implies compactness in the weak-star topology. Since the objective is a bounded linear functional and in particular weak-star lower semicontinuous we obtain the existence of a minimizer by standard reasoning.
\end{proof}\noindent

We mention that our existence proof implicitely uses additional regularity induced by the gradient state constraint as done also in \cite{wollner1}, however we do not discuss further regularity issues of the solution, which is more involved and possibly not to be expected with the nonsmooth terms in the objective functional. 

We observe that the objective as well as all terms involved in the constraints of the optimal control problem are one-homogeneous. This allows us to renormalize the unknown and in this way eliminate the safety constraint. This motivates the following simplified version:
\begin{eqnarray*}
  \mathbf{(P_{\epsilon})} \qquad 
&- \int_{\Omega_t}\! \langle A \boldsymbol I, \boldsymbol e \rangle \, \mathrm{d} \boldsymbol x& \rightarrow \min_{{\boldsymbol I} \in {\mathcal D}(\Gamma)} \\ 
& \text{subject to } &
 \omega | A \boldsymbol I |  \leq \epsilon 
\end{eqnarray*}
%In the Appendix, we show objective and residual convergence results for the simplified optimization problem $\mathbf{(P_{\epsilon})}$. \newline \noindent
By renormalizing a minimizer of $\mathbf{(P_{\epsilon})}$ we can obtain a minimizer for the full optimization problem $\mathbf{(P^2_{\delta})}$ with safety constraint for the control, which of course only holds if the total variation of the minimizer is finite:
\begin{theorem}\label{PeQe}
Let $\epsilon > 0$ be a threshold parameter and $\boldsymbol I \in \mathcal{M}(\Gamma) \cap H_{\diamond}^{-\frac{1}{2}}(\Gamma)$ be a minimizer of the simplified minimization problem $\mathbf{(P_{\epsilon})}$. Then $\tilde{\boldsymbol I} = \frac{4 \boldsymbol I}{\left\| \boldsymbol I\right\|_{\mathcal{M}(\Gamma)}}$ is a minimizer of $\mathbf{(P^2_{\delta})}$ with $\delta = \frac{4 \epsilon}{\left\| \boldsymbol I\right\|_{\mathcal{M}(\Gamma)}}$.
\end{theorem}
\begin{proof}\noindent
An easy calculation reveals that $\tilde{\boldsymbol I} = \frac{4 \boldsymbol I}{\left\| \boldsymbol I \right\|_{\mathcal{M}(\Gamma)}}$ fulfills the control and state constraint of $\mathbf{(P^2_{\delta})}$ with $\delta =$ $\frac{4 \epsilon}{\left\| \boldsymbol I \right\|_{\mathcal{M}(\Gamma)}}$. Assume that $\tilde{\boldsymbol I}$ is not a minimizer of $\mathbf{(P^2_{\delta})}$. Then there exists ${\boldsymbol I}_*$ with smaller functional value, and 
$\frac{\left\| \boldsymbol I \right\|_{\mathcal{M}(\Gamma)}}{4 \boldsymbol I} {\boldsymbol I}_*$ is feasible for $\mathbf{(P^2_{\delta})}$  with lower functional value than  ${\boldsymbol I}$, which contradicts the optimality of the latter.
\end{proof}\newline \noindent
At a first glance the usefulness of Theorem \ref{PeQe} might appear limited, since the relation between the constraint bounds $\epsilon$ and $\delta$ is quite implicit, depending on the minimizer of the second problem. However, since there is no natural choice of the bound in either case one can choose $\delta$ or $\epsilon$ in appropriate ranges equally well. It turns out however that $\mathbf{(P_{\epsilon})}$ is somewhat easier in particular with respect to numerical solutions. Moreover, it offers a better position to introduce sparsity constraints on the control by additional penalization, which is the subject of the following discussion. 

\subsection{Penalized Problem Formulation}

In order to control the applied current pattern at the fixed electrodes, the minimization problem $\mathbf{(P_{\epsilon})}$ is extended by two penalties. While an $L_2$ term $\alpha \int_{\Gamma}\! \boldsymbol I^2 \, \mathrm{d} \boldsymbol x$ is introduced to penalize the energy of the applied current pattern, an $L_1$ penalty $\beta\left\| \boldsymbol I \right\|_{\mathcal{M}(\Gamma)}$ is used to minimize the number of active electrodes in the minimization procedure. The latter has to be reinterpreted in the same way as the safety constraint in terms of Radon measures. This leads to the following minimization problem 
\begin{eqnarray*}
  \mathbf{(P^{\alpha,\beta}_{\epsilon})} \qquad 
&- \int_{\Omega_t}\! \langle A \boldsymbol I, \boldsymbol e \rangle \, \mathrm{d} \boldsymbol x &+ \alpha \int_{\Gamma}\! \boldsymbol I^2 \, \mathrm{d} \boldsymbol x + \beta\left\| \boldsymbol I \right\|_{\mathcal{M}(\Gamma)} \rightarrow \min_{{\boldsymbol I} \in {\mathcal D}(\Gamma)} \\ 
& \text{subject to } &
 \omega | A \boldsymbol I |  \leq \epsilon 
\end{eqnarray*}
We call the minimization problem $\mathbf{(P^{\alpha, \beta}_{\epsilon})}$ with $\alpha > 0, \beta = 0$ and $\alpha = 0, \beta > 0$ to be the $L_2$ regularized optimization procedure (L2R) and the $L_1$ regularized optimization procedure (L1R), respectively. If both penalties are active, the penalty term is also known as elastic net. \newline \noindent
Since the $L^2$ penalty adds strict convexity to the problem, we can obtain uniqueness of a minimizer together with existence via analogous arguments as above:
\begin{theorem}
Let $\mathcal{D}(\Gamma)$ be weakly closed in $H_{\diamond}^{-\frac{1}{2}}$, $0 \in \mathcal{D}(\Gamma)$, and let $\alpha > 0$ and $\beta = 0$. Then there exists a unique minimizer $\boldsymbol I \in L^2(\Gamma)$ to the L2R constrained optimization problem $\mathbf{(P^{\alpha,0}_{\epsilon})}$. In the latter statement, $\beta = 0$ can even be generalized to  $\beta \geq 0$. 
\end{theorem}\newline \noindent
For the problem with pure $L^1$-penalization the proof of existence is almost identical to the one of Theorem \ref{PeQe} and we hence conclude:
\begin{theorem}
Let $\mathcal{D}$ be weakly closed in $H_{\diamond}^{-\frac{1}{2}}$ and let $\alpha = 0$ and $\beta > 0$. Then there exists at least one minimizer $\boldsymbol I \in \mathcal{M}(\Gamma) \cap \mathcal{D}(\Gamma)$ to the L1R constrained optimization problem $\mathbf{(P^{0,\beta}_{\epsilon})}$.
\end{theorem} 

\subsection{Discrete Problem}
For numerical discretization of the considered optimization problem we use the finite element method. The solution of the forward problem is approximated by looking for $\Phi = \sum_{k=1}^M \alpha_k \Phi_k$ with the $\Phi_k$ being edge-based finite element basis functions vanishing on $\Gamma_D$.      
$$
\int_\Omega \sigma \nabla \Phi \cdot \nabla \Psi~dx = \int_\Gamma {\boldsymbol I}~\Psi ~d\sigma,  
$$
for all $\Psi$ in the span of $\{\Phi_k\}$. Using standard techniques this allows to verify the existence and uniqueness of the discrete solution for arbitrary $\boldsymbol I \in H_{\diamond}^{-\frac{1}{2}}(\Gamma)$ and to introduce a discretized solution operator 
$$\tilde A: H_{\diamond}^{-\frac{1}{2}}(\Gamma) \rightarrow L^2(\Omega)^3, {\boldsymbol I} \mapsto \sigma \nabla \Phi. $$  
Let $T_i$, $i=1,\ldots,N$ denote the volume elements in the finite element discretization (tetrahedra or cuboids) and denote for a function $u \in L^2(\Omega)^3$ the local mean value by  $$ \overline{u}_i = \frac{1}{|T_i|} \int_{T_i} u(x) ~dx. $$
Finally we discretize $\boldsymbol I = \sum_{j=1}^{S} I_j \phi_j$ with $\phi_j$ being local basis functions on $\Gamma$. From the fact that ${\boldsymbol I}$ has mean zero we can eliminate $I_S$. 
Now we introduce mappings $A_1$ and $A_2$ incorporating the discretizations:
\begin{eqnarray*}
A_1: \mathbb{R}^{S-1} (\Gamma) \rightarrow \mathcal{D}(\Gamma), \quad\boldsymbol I_S = (I_i) \mapsto \boldsymbol I = \sum_{i=1}^{S} I_i \phi_i\\
A_2:  L^2(\Omega)^3 \rightarrow \mathbb{R}^{3N}(\Omega), \quad \sigma \nabla \Phi  \mapsto (\overline{(\sigma \nabla \Phi)_i } )_{i=1,\ldots,N}
\end{eqnarray*}
and define the discretized operator $A_2 \circ \tilde A \circ A_1 = B$, represented by a matrix in $ \mathbb{R}^{3N \times (S-1)}$. 
\newline 
\noindent 
As the target vector $\boldsymbol e$ is only defined in the target region, we define a continuous mapping $\tilde{\boldsymbol e}$: $\Omega_t \mapsto  \mathbb{R}^{3N}(\Omega)$
\begin{align*}
\tilde{\boldsymbol e}_i=\left\{\begin{array}{cl} \boldsymbol e, & \mbox{if } T_i \subset \Omega_t \\ 0, & \mbox{otherwise } \end{array}\right.
\end{align*}
and replace $\langle B \boldsymbol I_S , \boldsymbol e \rangle$ by $\langle B \boldsymbol I_S , \tilde{\boldsymbol e} \rangle$, since $\langle B \boldsymbol I_S , \tilde{\boldsymbol e} \rangle$ = 0 in $\Omega \setminus \Omega_t$. 
We now introduce the discretized optimization problem $\mathbf{(\bar{P}^{\alpha,\beta}_{\epsilon})}$ as
\begin{eqnarray*}
  \mathbf{(\bar{P}^{\alpha,\beta}_{\epsilon})} \qquad 
&- \langle B \boldsymbol I_S, \tilde{\boldsymbol e} \rangle &+ \alpha \langle \boldsymbol I_S, \boldsymbol I_S\rangle + \beta \left\|\boldsymbol I_S \right\|_1 \rightarrow \min_{{\boldsymbol I}_s \in \mathbb{R}^{S-1}} \\ 
& \text{subject to } &
 \omega_i |  (B \boldsymbol I_S)_i |  \leq \epsilon,
\end{eqnarray*} 
where $\omega_i$ is a constant approximation of $\omega$ in $T_i$ (in our examples defined piecewise constant anyway).
\newline
The existence (and potential uniqueness) of minimizers can now be verified exactly as for the continuous case above, defining an appropriate discrete version of the $H^{-\frac{1}2}$-norm on 
$$  {\mathcal D}(\Gamma) = \{\boldsymbol I = \sum_{i=1}^{S} I_i \phi_i\},  $$
e.g. via 
$$ \Vert \boldsymbol I \Vert = \sup \{ \int_\Gamma {\boldsymbol I} \Phi ~d\sigma ~|~
\Phi \in \mbox{~span~}\{\Phi_k\}, \Vert \Phi \Vert_{H^1(\Omega)} \leq C_T \} .$$ 
For the sake of brevity we do not discuss the issue of, which can, e.g., be shown via $\Gamma$-convergence of the functionals involved.
\section{Numerical Optimization}
For numerical solution of the corresponding discretized problem we employ the alternating direction method of multipliers (ADMM) \cite{Boyd2011}. The ADMM is a variant of the 
\textit{Augmented Lagrangian method}, a class for solving constrained optimization problems. The method combines important convergence properties (no strict convergence or finiteness is required) and the decomposability of the dual ascent method \cite{Boyd2011}. \newline
To solve the discretized optimization problem $\mathbf{(\bar{P}^{\alpha, \beta}_{\epsilon})}$, we substitute $B\boldsymbol I_S = \boldsymbol y \in \mathbb{R}^{3N}$ and $\boldsymbol I_S = \boldsymbol z \in \mathbb{R}^{S-1}$ and obtain the following Lagrangian $L_{\mu_1,\mu_2}(\boldsymbol I_S,\boldsymbol y,\boldsymbol z,\boldsymbol p_1,\boldsymbol p_2)$ which is to be minimized
\begin{eqnarray*}
 L_{\mu_1,\mu_2}(\boldsymbol I_S,\boldsymbol y,\boldsymbol z,\boldsymbol p_1,\boldsymbol p_2) & = & \alpha \langle\boldsymbol z,\boldsymbol z\rangle + \beta  \left\|\boldsymbol z \right\|_1 + \frac{\mu_1}{2} \langle\boldsymbol z-\boldsymbol I_S,\boldsymbol z-\boldsymbol I_S\rangle + \langle \boldsymbol z-\boldsymbol I_S, \boldsymbol p_1\rangle  \\
 &  &- \langle \boldsymbol y , \tilde{\boldsymbol e}\rangle + \frac{\mu_2}{2} \langle \boldsymbol y - B\boldsymbol I_S, \boldsymbol y - B\boldsymbol I_S\rangle + \langle \boldsymbol y - B\boldsymbol I_S, \boldsymbol p_2 \rangle \\
 & &\text{subject to} \; \omega_i |\boldsymbol y_i| \leq \epsilon
 \end{eqnarray*}
with $\mu_1, \mu_2 \in \mathbb{R} $ and $\boldsymbol p_1 \in \mathbb{R}^{S-1}, \boldsymbol p_2 \in \mathbb{R}^{3N}$ being the \textit{augmented Lagrangian parameters} and the \textit{dual variables}, respectively \cite{Boyd2011}. \newline
\textbf{$I_S$-step}: Rearranging and neglecting terms without $\boldsymbol I_S$, putting the integrals together and expanding leads to a quadratic minimization problem
\begin{eqnarray*}
 \frac{\mu_1}{2} (\langle\boldsymbol I_S^{k+1},\boldsymbol I_S^{k+1}\rangle & - & 2 \langle\boldsymbol z^k,\boldsymbol I_S^{k+1}\rangle + \frac{2}{\mu_1} \langle\boldsymbol I_S^{k+1},\boldsymbol p_1^k\rangle) \\
&+& \frac{\mu_2}{2} (-2\langle\boldsymbol y^k,B\boldsymbol I_S^{k+1}\rangle + \langle B\boldsymbol I_S^{k+1}, B\boldsymbol I_S^{k+1}\rangle + \frac{2}{\mu_2}\langle B\boldsymbol I_S^{k+1},\boldsymbol p_2^k\rangle)
\end{eqnarray*}
Differentiating with respect to $\boldsymbol I_S^{k+1}$ and setting the equation system to be zero results in
\begin{eqnarray*}
& &(\mu_1 Id + \mu_2 B^{tr} B) \boldsymbol I_S^{k+1} - (\mu_1 \boldsymbol z^k - \boldsymbol p_1^k + \mu_2 B^{tr} \boldsymbol y^k - B^{tr} \boldsymbol p_2^k) = 0 \\
& & \Rightarrow \boldsymbol I_S^{k+1} = (\mu_1Id + \mu_2 B^{tr}B)^{-1}(\mu_1 \boldsymbol z^k - \boldsymbol p_1^k + \mu_2 B^{tr} \boldsymbol y^k - B^{tr} \boldsymbol p_2^k)
\end{eqnarray*}
\textbf{$y$-step}: Rearranging and neglecting terms and putting the integrals together results in
\begin{eqnarray*}
& &\frac{\mu_2}{2} \langle \boldsymbol y^{k+1} - {B\boldsymbol I_S^{k+1} - \frac{1}{\mu_2} \boldsymbol p_2^k - \frac{1}{\mu_2} \tilde{\boldsymbol e}}\ ,\ \boldsymbol y^{k+1} - {B\boldsymbol I_S^{k+1} - \frac{1}{\mu_2} \boldsymbol p_2^k - \frac{1}{\mu_2} \tilde{\boldsymbol e}}\rangle\\
& & \text{subject to} \; \omega_i |\boldsymbol y_i| \leq \epsilon
\end{eqnarray*}
which can be solved analytically as follows
\begin{equation*}
\boldsymbol y^{k+1}_i =\left\{\begin{array}{cl}  \frac{\epsilon}{\omega_i} \frac{(B\boldsymbol I_S^{k+1} + \frac{1}{\mu_2}\boldsymbol p_2^k + \frac{1}{\mu_2}\tilde{\boldsymbol e})_i}{|(B\boldsymbol I_S^{k+1} + \frac{1}{\mu_2} \boldsymbol p_2^k + \frac{1}{\mu_2}\tilde{\boldsymbol e})_i|}, & \mbox{if } |(B\boldsymbol I_S^{k+1} + \frac{1}{\mu_2}\boldsymbol p_2^k + \frac{1}{\mu_2}\tilde{\boldsymbol e})_i| > \frac{\epsilon}{\omega_i} \\ 
(B\boldsymbol I_S^{k+1} + \frac{1}{\mu_2}\boldsymbol p_2^k + \frac{1}{\mu_2}\tilde{\boldsymbol e})_i, & \mbox{otherwise } \end{array}\right.
\end{equation*}
\textbf{$z$-step}: Rearranging the terms, leaving out the terms without $\boldsymbol z$ and putting the integrals together leads to
\begin{eqnarray*}
& &\alpha \langle \boldsymbol z^{k+1}, \boldsymbol z^{k+1} \rangle + \beta \left\| \boldsymbol z^{k+1} \right\|_1 + \frac{\mu_1}{2} \langle \boldsymbol z^{k+1}-\boldsymbol I_S^{k+1}, \boldsymbol z^{k+1}-\boldsymbol I_S^{k+1} \rangle + \langle \boldsymbol z^{k+1}-\boldsymbol I_S^{k+1} , \boldsymbol p_1^k \rangle \\
& = &\alpha \langle \boldsymbol z^{k+1}, \boldsymbol z^{k+1} \rangle \\
&+ & \langle \underbrace{\sqrt{\frac{\mu_1}{2}}Id}_{:= \tilde{Id}} \boldsymbol z^{k+1}- \sqrt{\frac{\mu_1}{2}} \boldsymbol I_S^{k+1} - \frac{1}{\sqrt{2 \mu_1}} \boldsymbol p_1^k + \frac{1}{\sqrt{2 \mu_1}} \beta , \sqrt{\frac{\mu_1}{2}}Id \boldsymbol z^{k+1}- \sqrt{\frac{\mu_1}{2}} \boldsymbol I_S^{k+1} - \frac{1}{\sqrt{2 \mu_1}} \boldsymbol p_1^k + \frac{1}{\sqrt{2 \mu_1}} \beta  \rangle
\end{eqnarray*}
The solution to this equation is given as
\begin{equation*}
 \boldsymbol z^{k+1} = (\tilde{Id}^{tr} \tilde{Id} + \alpha Id)^{-1} \tilde{Id}^{tr} (\sqrt{\frac{\mu_1}{2}}\boldsymbol I_S^{k+1} + \frac{1}{\sqrt{2 \mu_1}} \boldsymbol p_1^k - \frac{1}{\sqrt{2 \mu_1}} \beta )
\end{equation*}
\textbf{Dual update}: Finally, the dual variables are updated
\begin{eqnarray*}
 \boldsymbol p_1^{k+1} & = & \mu_1(\boldsymbol p_1^k + \boldsymbol I_S^{k+1} - \boldsymbol z^{k+1})\\
 \boldsymbol p_2^{k+1} & = & \mu_2(\boldsymbol p_2^k + B\boldsymbol I_S^{k+1} - \boldsymbol y^{k+1})
\end{eqnarray*}
We mention that a complete convergence analysis can be carried out following the arguments in \cite{Boyd2011}.
Algorithm \ref{AlgotCSOpt} depicts an overview of the optimization steps for the numerical solution of the discretized problem $\mathbf{(\bar{P}^{\alpha, \beta}_{\epsilon})}$. 
\rev{Note that the Euclidean norm was used for the stopping criterion in Step 3.}
\begin{algorithm}[t]
\caption{Algorithm for the discretized minimization problem $\mathbf{(\bar{P}^{\alpha, \beta}_{\epsilon})}$}
\begin{algorithmic}[1] \label{AlgotCSOpt}
\STATE Input: $B, \epsilon, \mu_1, \mu_2, \alpha, \beta, \omega, \tilde{\boldsymbol e}$, $\boldsymbol I_{prev}$, $\boldsymbol I_S^0$, $\boldsymbol p_1^0$, $\boldsymbol p_2^0$, $\boldsymbol z^0$, $\boldsymbol y^0$, $N$, $TOL$
\STATE $k=0$
\WHILE{$ k < 3 \; or \; \left\| \boldsymbol I_S^{k} - \boldsymbol I_{prev} \right\| > TOL$}
\STATE $\boldsymbol I_{prev} = \boldsymbol I_S^{k}$
\STATE $\boldsymbol I_S^{k+1} = (\mu_1Id + \mu_2 B^{tr}B)^{-1}(\mu_1 \boldsymbol z^k - \boldsymbol p_1^k + \mu_2 B^{tr} \boldsymbol y^k - B^{tr} \boldsymbol p_2^k)$ 
\FOR{$i=1, \cdots, 3N$}
\IF{$|(B\boldsymbol I_S^{k+1} + \frac{1}{\mu_2}\boldsymbol p_2^k + \frac{1}{\mu_2}\tilde{\boldsymbol e})_i| > \frac{\epsilon}{\omega_i}$}
\STATE $\boldsymbol y_i^{k+1} = \frac{\epsilon}{\omega_i} \frac{(B\boldsymbol I_S^{k+1} + \frac{1}{\mu_2}\boldsymbol p_2^k + \frac{1}{\mu_2}\tilde{\boldsymbol e})_i}{|(B\boldsymbol I_S^{k+1} + \frac{1}{\mu_2}\boldsymbol p_2^k + \frac{1}{\mu_2}\tilde{\boldsymbol e})_i|}$
\ELSE
\STATE $\boldsymbol y_i^{k+1} = (B\boldsymbol I_S^{k+1} + \frac{1}{\mu_2}\boldsymbol p_2^k + \frac{1}{\mu_2}\tilde{\boldsymbol e})_i$
\ENDIF
\ENDFOR
\STATE $\boldsymbol z^{k+1} = (\tilde{Id}^{tr} \tilde{Id} + \alpha Id)^{-1} \tilde{Id}^{tr} (\sqrt{\frac{\mu_1}{2}} \boldsymbol I_S^{k+1} + \frac{1}{\sqrt{2 \mu_1}} \boldsymbol p_1^k - \frac{1}{\sqrt{2 \mu_1}} \beta)$
\STATE $\boldsymbol p_1^{k+1} = \mu_1(\boldsymbol p_1^k +\boldsymbol  I_S^{k+1} - \boldsymbol z^{k+1})$
\STATE $ \boldsymbol p_2^{k+1} = \mu_2(\boldsymbol p_2^k + B\boldsymbol I_S^{k+1} - \boldsymbol y^{k+1})$
\STATE $k = k+1$  
\ENDWHILE
\STATE $\delta = \frac{4 \epsilon}{\left\| \boldsymbol I_S^{k+1} \right\|_{\mathcal{M}(\Gamma)}}$
\RETURN $\boldsymbol I_S^{k+1},\ \delta,\ k,\ \left\| \boldsymbol I_S^{k+1} \right\|_{\mathcal{M}(\Gamma)}$
\end{algorithmic}
\end{algorithm} 

\section{Results and Discussion}
\begin{table}[!t]
	\centering
  \renewcommand{\arraystretch}{1.3}
\caption{Tangential target vector $\tilde{\boldsymbol e}$: Averaged (over 924 targets) scaled threshold value $\delta$, the number of iterations $k$ and the scaling factor $\left\| \boldsymbol I \right\|_{\mathcal{M}(\Gamma)}$ for tangential target vectors $\tilde{\boldsymbol e}$ and different input values $\epsilon$.}\label{TabEpsTang}
	\begin{tabular}{cccccc}
			\hline
 			{\bfseries $ \epsilon$} & 0.0001 & 0.0005 & 0.001 & 0.005 & 0.01 \\
			\hline \hline
			{\bfseries $ \delta \left[ A m^{-2} \right]$} & 0.00332 & 0.0085 & 0.0154 & 0.0175 & 0.0208 \\
			{\bfseries Iterations $k$} & 132 & 53 & 33 & 203 & 143\\
			{\bfseries $\left\| \boldsymbol I \right\|_{\mathcal{M}(\Gamma)}$} & 120.15 & 233.96 & 352.279 & 1312.2 & 1919.93 \\
			\hline
	\end{tabular}
\end{table}\noindent
We generated a highly realistic geometry-adapted six-compartment (skin, skull compacta, skull spongiosa, CSF, gray and white matter) finite element head model from a T1- and a 
T2-weighted magnetic resonance image (MRI). For the compartments skin, skull compacta, skull spongiosa, CSF and gray matter we employed conductivity values $\sigma$ = \SI{0.43}{S m^{-1}}, \SI{0.007}{S m^{-1}}, \SI{0.025}{S m^{-1}}, \SI{1.79}{S m^{-1}} and \SI{0.33}{S m^{-1}}, respectively \cite{Haueisen2002, Ramon2004, Wagner2014}. For the modeling of the white matter anisotropy, a diffusion tensor MRI was used and the effective medium approach was applied. The effective medium approach assumes a linear relationship between the effective electrical conductivity tensor and the effective water diffusion tensor in white matter \cite{Tuch2001}, resulting in a mean conductivity value $\sigma = \SI{0.14}{S m^{-1}}$ for the white matter compartment \cite{Wagner2014}. As the current density amplitudes in the skin and CSF compartments are much stronger when compared to the brain compartments \cite{Wagner2014}, we only visualized the current densities in the brain to enable best visibility of the cortical current flow pattern. The software package SCIRun was used for visualization \cite{Dannhauer2012}. \newline
Using this highly realistic volume conductor model we comprehensively examine the optimized current densities in the brain with regard to focality, target intensity and orientation. Firstly, the 
maximal current density in non-target regions is investigated for the discretized optimization problem $\mathbf{(\bar{P}^{\alpha,\beta}_{\epsilon})}$ and four target vectors are used to calculate optimized stimulation protocols. For all simulations, the target area $\Omega_t$ always consists of the elements corresponding to the target vectors and only the volume conductor elements in the brain are used for optimization, i.e., $\omega = 1$ only in the brain compartments and $\omega \ll 1$ in the target region and in the CSF, skin and skull compartments. \newline
In order to investigate the current flow field that is induced by an optimized standard bipolar electrode montage, we use a maximum two electrode (M2E) approach. This approach stimulates only the main positive (anode) and the main negative electrode (cathode) of the L1R optimized stimulation protocols with a total current of \SI{1}{mA}. In order to quantify the optimized current flow fields, we calculate the current densities in the direction of the target vectors ($CD_t$) and the percentage of current density that is oriented parallel to the target vector ($PAR$), as shown in columns 4 and 5 in Table \ref{Quantification}, respectively.

\subsection{Maximal current density in non-target regions}
In this section, the L1R approximated discretized optimization problem $\mathbf{(\bar{P}^{0,\beta}_{\epsilon})}$ with $\beta$ = 0.001 is used to calculate optimized stimulation protocols for a set of 924 tangential (parallel to the inner skull surface) and radial (perpendicular to the inner skull surface) target vectors and Theorem \ref{PeQe} is applied to estimate the averaged maximal current density in non-target regions as shown in Tables \ref{TabEpsTang} and \ref{TabEpsRad}. Furthermore, the number of iterations until a minimum was found and the scaling factor $\left\| \boldsymbol I\right\|_{\mathcal{M}(\Gamma)}$ is depicted. 
\rev{As can be seen in the tables, for the threshold value $\epsilon$ = 0.001, the safety value $\delta$ ensures that the maximal current density in the brain compartment is not dangerous for the subject. Furthermore, the number of iterations $k$ until a minimum was found is lowest for $\epsilon$ = 0.001. For this reason, a threshold value of $\epsilon$ = 0.001 is thus used in this study as it allows safe and well-targeted stimulation and fast and robust computation of the stimulation protocol.} 
\rev{Note that the number of iterations until a minimum is found increases with decreasing mesh size. However, because the maximal resolution of state-of-the-art whole-head MRI sequences is 1mm, we did not investigate mesh sizes smaller than 1mm.}

\begin{table}[!t]
	\centering
  \renewcommand{\arraystretch}{1.3}
\caption{Radial target vector $\tilde{\boldsymbol e}$: Averaged (over 924 targets) scaled threshold value $\delta$, the number of iterations $k$ and the scaling factor $\left\| \boldsymbol I \right\|_{\mathcal{M}(\Gamma)}$ for radial target vectors 
$\tilde{\boldsymbol e}$ and different input values $\epsilon$.} \label{TabEpsRad}
	\begin{tabular}{cccccc}
			\hline
 			{\bfseries $ \epsilon $} & 0.0001 & 0.0005 & 0.001 & 0.005 & 0.01 \\
			\hline \hline
			{\bfseries $ \delta \left[ A m^{-2} \right]$} & 0.0043 & 0.0114 & 0.015 & 0.023 & 0.025 \\
			{\bfseries Iterations $k$} & 134 & 36 & 35 & 68 & 118\\
			{\bfseries $\left\| \boldsymbol I \right\|_{\mathcal{M}(\Gamma)}$} & 92.03 & 176.08 & 267.06 & 887.08 & 1597.73 \\
			\hline
	\end{tabular}
\end{table}

\subsection{Mainly tangential target vector}\label{tangsingle}
\begin{figure}[t!]
\begin{center}
\includegraphics*[width = \textwidth,keepaspectratio]{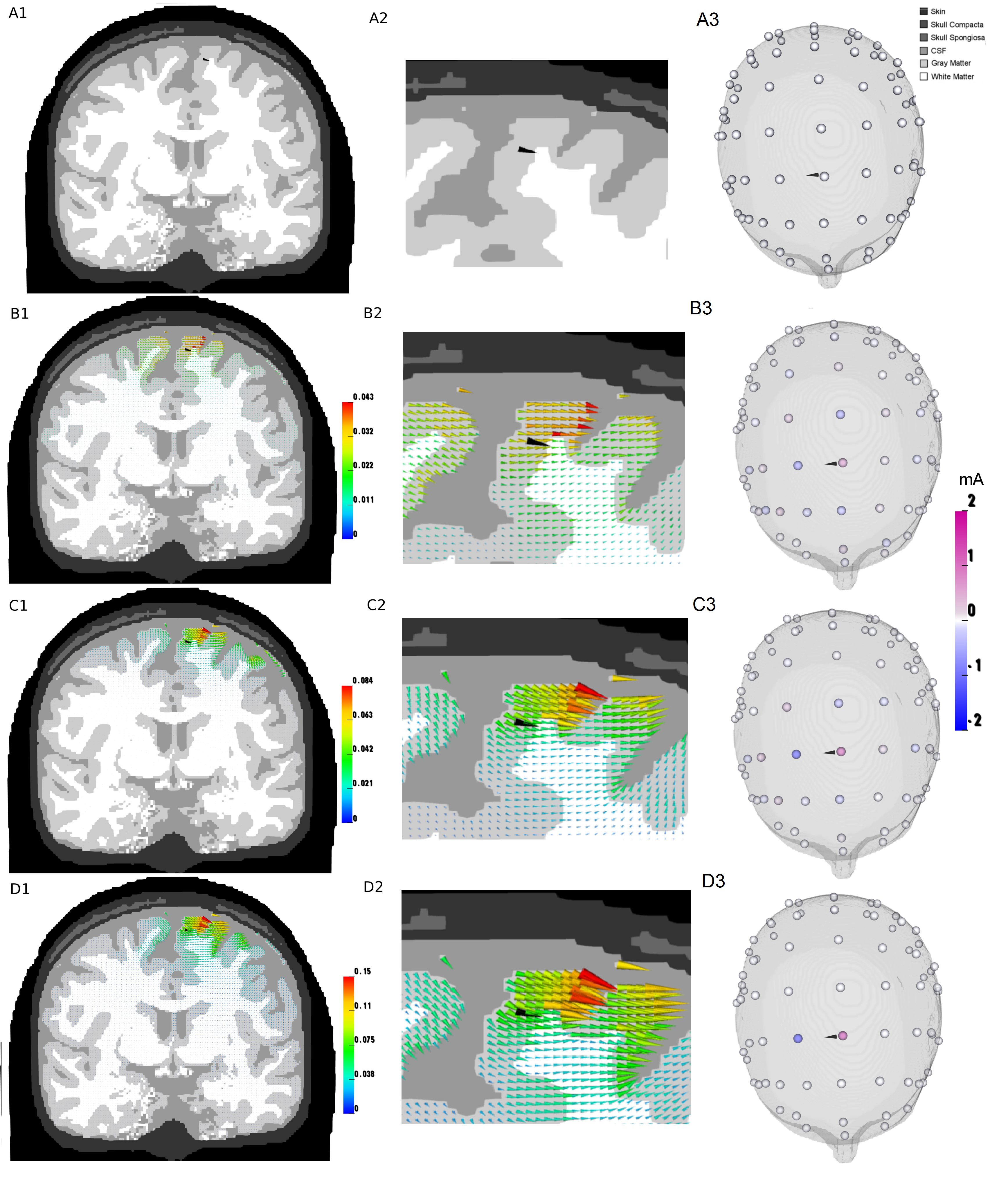}
\end{center}
\caption{A mainly tangential target vector (Figs. -A1 and -A3): Figures -B -C and -D present optimization results for the L2R, L1R and M2E approach, respectively. \rev{The optimized current densities in the brain compartment (different scaling of rows), in a zoomed region of interest and the corresponding stimulation protocols (same scaling) are shown in Figures -1, -2 and -3, respectively.}}\label{FigTangOpt}
\end{figure}
Figure \ref{FigTangOpt} displays the optimized current densities (for better visibility, we used a different scaling for the different rows) and the corresponding stimulation protocols (same scaling) for a mainly tangential target vector as shown in Figures \ref{FigTangOpt}-A1 and -A2. As can be seen in Figures \ref{FigTangOpt}-B1 and \ref{FigTangOpt}-C1, the optimized current flow fields using the L2R and L1R approaches, respectively, show high focality and maximal current densities can be observed in the target region. While the L2R and L1R approaches lead to target current densities of \SI{0.022}{A m^{-2}} and \SI{0.038}{A m^{-2}} (Table \ref{Quantification}, second column), current densities in non-target brain regions are rather weak (Figs.\ref{FigTangOpt}-B1 and -C1). Overall, due to the rather widespread applied current pattern at the fixed electrodes (Fig.\ref{FigTangOpt}-B2), the L2R optimized current flow field (Fig.\ref{FigTangOpt}-B1) has smaller amplitude when compared to the L1R one (Fig.\ref{FigTangOpt}-C1). On the other hand, the L1R stimulation protocol (Fig.\ref{FigTangOpt}-C2) shows high focality with mainly two active electrodes, while only very weak compensating currents are injected at the neighboring electrodes. The M2E approach results in higher target current densities of \SI{0.071}{A m^{-2}} (Table \ref{Quantification}, second column). However, using the M2E approach, relatively strong current densities can also be noted in non-target regions, especially in pyramidal tract and deeper white matter regions (Fig.\ref{FigTangOpt}-D1). \newline
With a $PAR$ value of 86.4, 86.8 and 85.9 for L2R, L1R and M2E (Table \ref{Quantification}, fifth column), the current densities of all three approaches are mainly oriented parallel to the target vector, leading to $CD_t$ values, i.e., current densities along the target direction, of \SI{0.019}, \SI{0.033} and \SI{0.061}{A m^{-2}}, resp. (Table \ref{Quantification}, fourth column). While the $CD_t$ values between L1R and M2E are thus less than a factor of 2 different, the averaged current density amplitudes in non-target regions is about 5.3 times higher when using the M2E approach (Table \ref{Quantification}, third column). The L1R optimized current flow field thus shows significantly higher focality and also slightly better parallelity to the target vector in comparison to the bipolar electrode montage M2E. However, if no multi-channel tDCS stimulator is available, the M2E approach provides an optimized bipolar electrode montage for a mainly tangential target vector.

\begin{table}[!t]
	\centering
  \renewcommand{\arraystretch}{1.3}
\caption{Quantification of optimized current density. The averaged current density in the target area ($CD_a$, second column), the averaged current density in non-target regions
(third column), the inner product of current density and target vector ($CD_t$, fourth column) and the percentage of current density that is oriented parallel to the target vector (PAR, fifth column) is displayed for different target vectors and methods (first column).} \label{Quantification}
	\begin{tabular}{ccccc}
			\hline
 			&\multicolumn{3}{c} {\bfseries  $\left[ A m^{-2} \right]$} & {\bfseries  $\left[ \% \right]$}\\
			{\bfseries Target } & {\bfseries $\frac{\int_{\Omega_t}\! |\boldsymbol B \boldsymbol I_S|\, \mathrm{d} \boldsymbol x}{|\Omega_t|}$} & {\bfseries $\frac{\int_{\Omega \setminus \Omega_t}\! |\boldsymbol B \boldsymbol I_S|\, \mathrm{d} \boldsymbol x}{|\Omega \setminus \Omega_t|}$}& {\bfseries $\frac{\int_{\Omega}\! \langle \boldsymbol B \boldsymbol I_S,\boldsymbol e\rangle \, \mathrm{d} \boldsymbol x}{|\Omega_t|}$} & {\bfseries $PAR = \frac{CD_t}{CD_a}$}  \\
			\hline \hline
tangential L2R & 0.022 & 0.00144 & 0.019 & 86.4\\
tangential L1R & 0.038 &  0.00151 & 0.033 & 86.8\\
tangential M2E & 0.071 & 0.0080 & 0.061 & 85.9\\ \hline
radial L2R & 0.026 &  0.00064 & 0.025 & 96.1\\
radial L1R & 0.045 &  0.00071 & 0.043 & 95.5\\
radial M2E & 0.063 & 0.0074 & 0.048 & 76.2\\ \hline
patch L2R & 0.025 & 0.00147 & 0.022 & 88.0\\
patch L1R & 0.037 & 0.00157 & 0.033 & 89.1 \\
patch M2E& 0.071 & 0.0080 & 0.062 & 87.3\\ \hline
deep L2R & 0.015 & 0.00345 & 0.013 & 86.7\\
deep L1R & 0.019 & 0.00249 & 0.018 & 94.7 \\
deep M2E & 0.052 & 0.01477 & 0.049 & 94.2\\
			\hline
	\end{tabular}
\end{table}

\begin{figure}[t!]
\begin{center}
\includegraphics*[width = \textwidth,keepaspectratio]{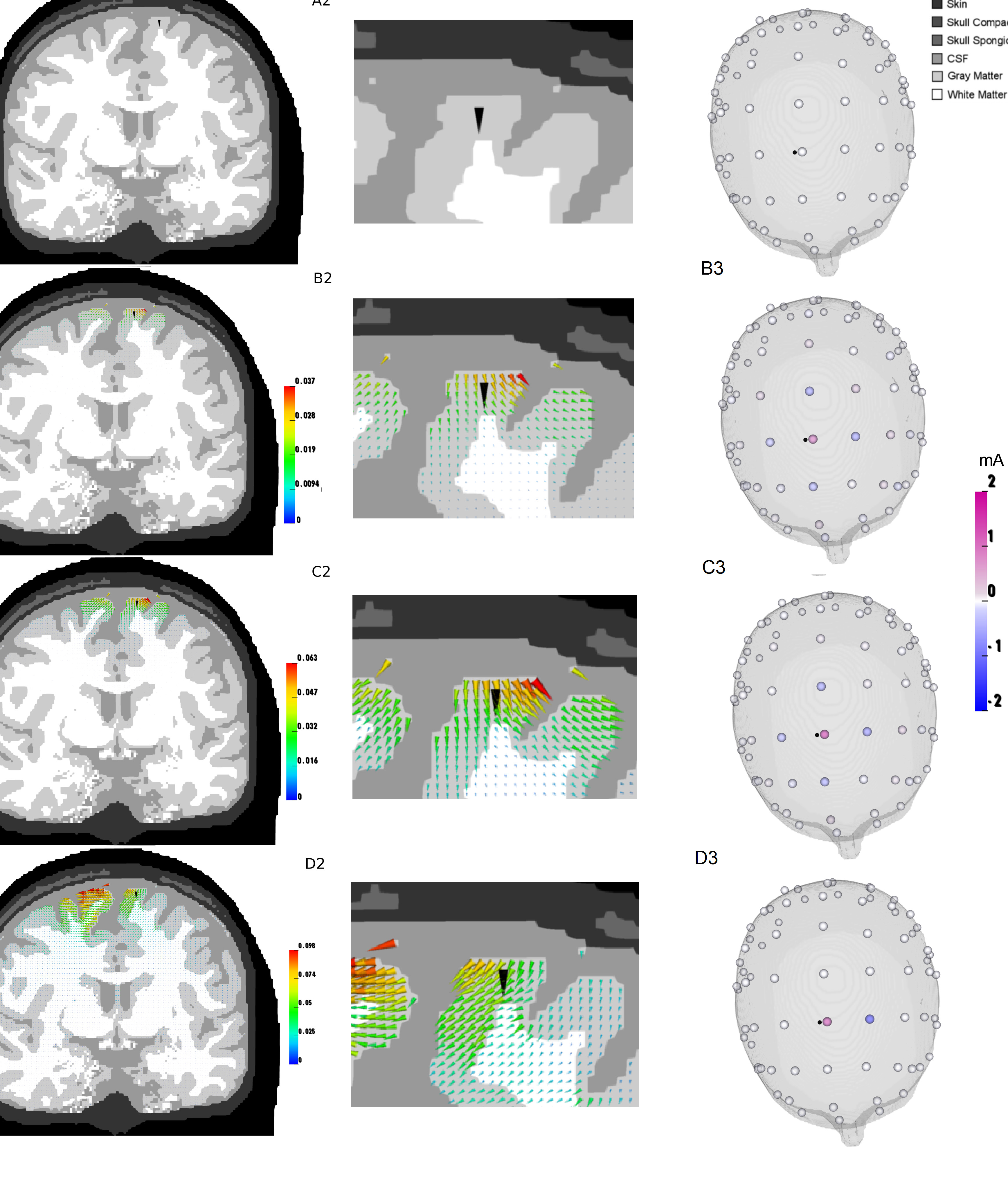}
\end{center}
\caption{A mainly radial target vector (Figs. -A1 and -A2): Figures -B -C and -D present optimization results for the L2R, L1R and M2E approach, respectively. \rev{The optimized current densities in the brain compartment (different scaling of rows), in a zoomed region of interest and the corresponding stimulation protocols (same scaling) are shown in Figures -1, -2 and -3, respectively.}}\label{FigRadOpt}
\end{figure}

\subsection{Mainly radial target vector}\label{radsingle}
\begin{figure}[t!]
\begin{center}
\includegraphics*[width = \textwidth,keepaspectratio]{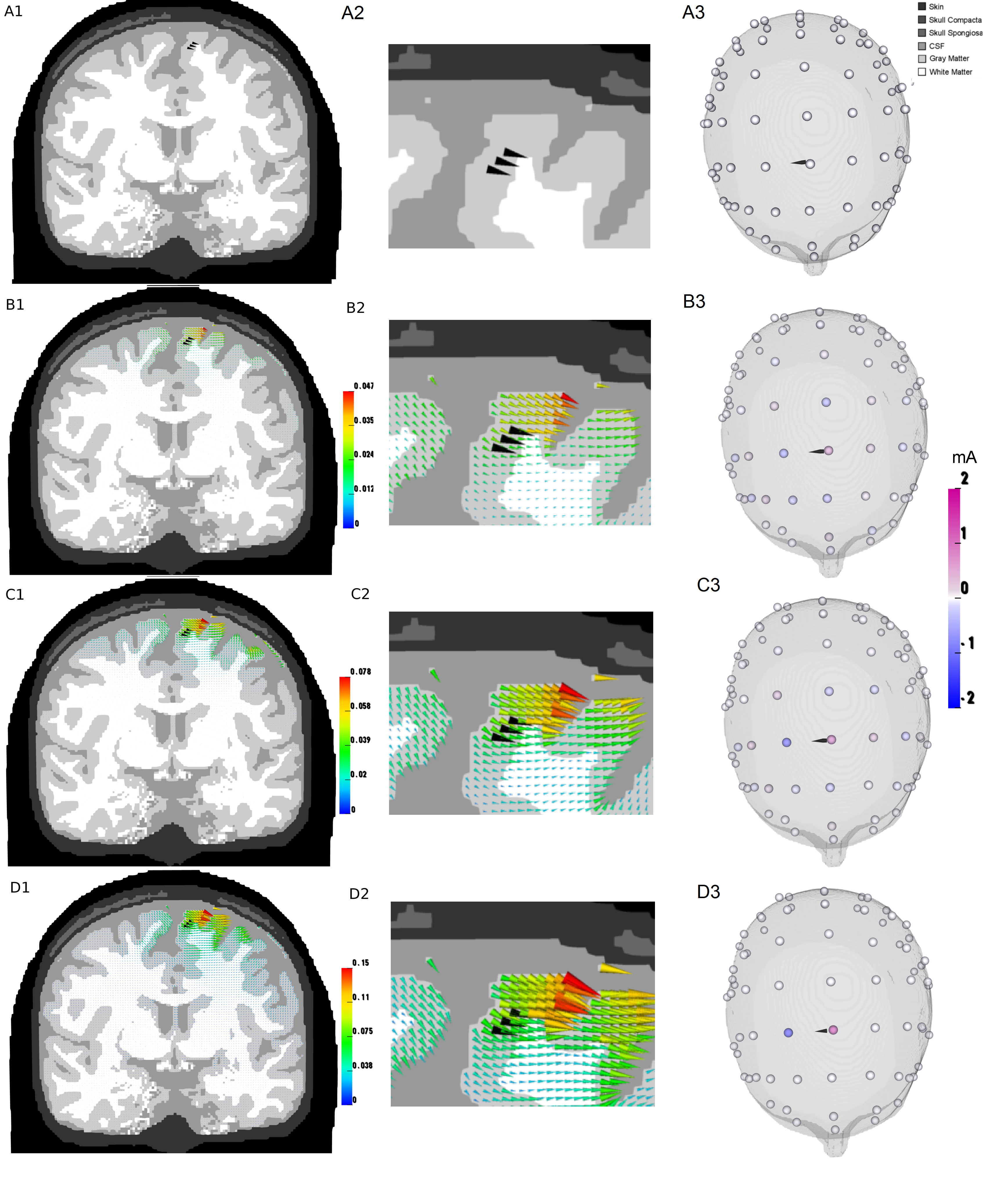}
\end{center}
\caption{An extended target region consisting of 3 tangential target vectors (Figs. -A1 and -A2):  Figures -B -C and -D present optimization results for the L2R, L1R and M2E approach, respectively. \rev{The optimized current densities in the brain compartment (different scaling of rows), in a zoomed region of interest and the corresponding stimulation protocols (same scaling) are shown in Figures -1, -2 and -3, respectively.}}\label{FigPatch}
\end{figure}\noindent
The mainly radially oriented target is shown in Figures \ref{FigRadOpt}-A1 and -A2. Figures \ref{FigRadOpt}-B1 and \ref{FigRadOpt}-C1 depict the optimized current densities for the L2R and L1R approaches, resp.. As shown in Table \ref{Quantification} (Column 2), with a value of \SI{0.045}{A m^{-2}}, the target current density  for the L1R approach is more than a factor of 1.7 times stronger than with the L2R approach (\SI{0.026}{A m^{-2}}). Non-target regions show only weak current densities  (Figures \ref{FigRadOpt}-B1 and -C1). With a value of \SI{0.063}{A m^{-2}}, the M2E approach yields the largest target intensity (Table \ref{Quantification}, second column). However, for this approach, the maximal current density in the brain does not occur in the target region, but on a neighboring gyrus in between the stimulating electrodes and in mainly tangential direction (Figure \ref{FigRadOpt}-D1). Moreover, the M2E optimized current density shows again an overall much lower focality than the L2R and L1R current flow fields. \newline
The L2R stimulation protocol (Figure \ref{FigRadOpt}-B2) consists of an anode above the target, surrounded by four cathodes and a ring of very weak positive currents at the second neighboring electrodes, a distribution which might be best described by a sinc-function. As can be seen in Figure \ref{FigRadOpt}-C2, the L1R stimulation protocol is mainly composed of a main positive current at the electrode above the target region and four return currents applied to the surrounding electrodes. \newline
The L2R and L1R optimized current flow fields show high directional agreement with the target vector $\boldsymbol e$ ($PAR$ value above  \SI{95}{\percent}), with the L2R slightly outperforming the L1R approach (Table \ref{Quantification}, fifth column). With a $PAR$ value of  only \SI{76.2}{\percent} the current flow field in the M2E model shows much less directional agreement with the target vector. This results in $CD_t$ values of 0.025, 0.043 and \SI{0.048}{A m^{-2}} for the L2R, L1R and M2E approaches, resp.. In order to obtain higher $PAR$ values for the M2E approach, i.e., higher current densities along the target direction, the distance between anode and cathode might be enlarged, in line with Dmochowski and colleagues \cite{Dmochowski2011} who reported that the optimal bipolar electrode configuration for a radial target vector consists of an electrode placed directly over the target with a distant return electrode. Another interesting bipolar electrode arrangement for a radial target vector might consist of an anode over the target and a cathodal ring around the anode.

\begin{figure}[t!]
\begin{center}
\includegraphics*[width = \textwidth,keepaspectratio]{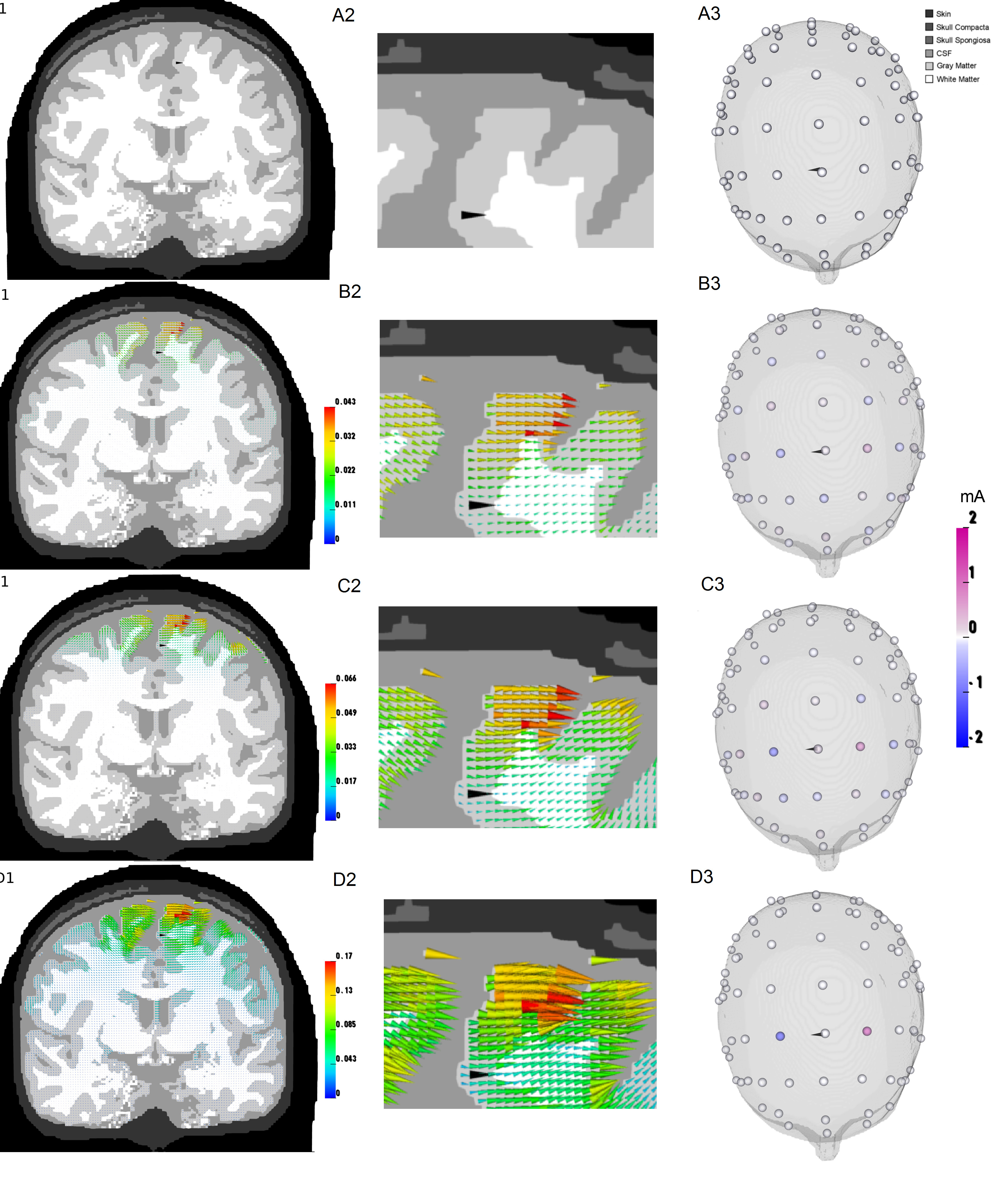}
\end{center}
\caption{A deeper target vector (Figs. -A1 and -A2): Figures -B -C and -D present optimization results for the L2R, L1R and M2E approach, respectively. \rev{The optimized current densities in the brain compartment (different scaling of rows), in a zoomed region of interest and the corresponding stimulation protocols (same scaling) are shown in Figures -1, -2 and -3, respectively.}}\label{FigDeep}
\end{figure}

\subsection{Extended mainly tangential target vector region}
In this section, an extended target area of \SI{3}{mm} $\times$ \SI{1}{mm} $\times$ \SI{1}{mm} is used for current density optimization. The target area is centered around the location of the superficial target vector that was also used in Section \ref{tangsingle} and the corresponding target vectors are selected to be mainly tangentially oriented (Figures \ref{FigPatch}-A1 and \ref{FigPatch}-A2). As can be seen in Figures \ref{FigPatch}-B1 and \ref{FigPatch}-C1, the optimized current flow fields of the L2R and L1R approaches show high focality. L2R and L1R yield an averaged target intensity of \SI{0.025}{A m^{-2}} and \SI{0.037}{A m^{-2}}, resp. (Table \ref{Quantification}, second column). Current density is not restricted to but focused to the target area, especially for the L2R approach (Fig. \ref{FigPatch}-B1).  In comparison to the superficial tangential target vector of Section \ref{tangsingle}, the L1R optimized averaged current density in the target area is decreased by about \SI{3}{\percent} (from \SI{0.038}{A m^{-2}} to \SI{0.037}{A m^{-2}}) and the optimized stimulation protocols are also very similar. Because the main two electrodes are taken from the L1R optimization, this implies that also the M2E stimulation protocol remains constant to the stimulation protocol from Section \ref{tangsingle}. In this way, the M2E approach yields an averaged target intensity of \SI{0.071}{A m^{-2}} (Table \ref{Quantification}, second column). \newline
With averaged $PAR$ values of 88.0, 89.1 and 87.3 (Table \ref{Quantification}, fifth column), the current densities are mainly oriented parallel to the target vectors with L1R performing best. The averaged current flow field intensity along the target direction is 0.022, 0.033 and \SI{0.062}{A m^{-2}} for the L2R, L1R and M2E approaches, resp. (Table \ref{Quantification}, fourth column). However, for M2E also non-target regions reach significant current densities, as clearly shown in Figure \ref{FigPatch}-D1. The averaged current density amplitudes in non-target regions is 0.00157 and \SI{0.0080}{A m^{-2}} for the L1R and M2E approaches (Table \ref{Quantification}, third column). The L1R optimized current flow field thus shows a factor of 5.1 higher focality in comparison to bipolar electrode montage M2E. However, the M2E approach provides an optimized bipolar electrode montage for an extended target area of tangential target vectors. 

\subsection{Deep and tangential target vector}
In the last simulation scenario, we investigate optimization for a deeper and mainly tangentially oriented target vector as shown in Figures \ref{FigDeep}-A1 and -A2. Figures \ref{FigDeep}-B1 and -C1 depict the optimized current density distributions when using L2R and L1R for optimization, resp.. For those approaches, the target current densities are 0.015 and \SI{0.019}{A m^{-2}}, resp.. With a value of \SI{0.052}{A m^{-2}}, which is more than 2.7 times the L1R value, the largest target intensity is, however, achieved with the M2E approach (Table \ref{Quantification}, second column).  
$CD_t$ values of 0.013, 0.018 and \SI{0.049}{A m^{-2}} lead to $PAR$ values of 86.7, 94.7 and 94.2 for the L2R, L1R and M2E approaches (Table \ref{Quantification}, fourth and fifth column). The target current densities are thus for all three approaches oriented mainly parallel to the target vector. \newline
The L2R and L1R stimulation protocols show high focality with mainly two active electrodes, while only weak compensating currents are injected at the neighboring electrodes (Figures  \ref{FigTangOpt}-B2 and -C2). Similar to Section \ref{tangsingle}, the compensating currents are stronger when using the L2R optimization procedure, leading to weaker target brain current densities as compared to the L1R stimulation. In order to enable current density to penetrate into deeper brain regions, the distance between the two main stimulating electrodes is larger when compared to the superficial mainly tangential target vector from Section \ref{tangsingle}, i.e., the electrode above the target region is not used for stimulation, while a more distant electrode is used as anode. \newline  
For all three approaches, strongest current density amplitudes in the brain compartment always occur at the CSF/brain boundary above the target region (Figures \ref{FigDeep}-B1,-C1 and -D1). This is due to the fact that the potential field $\nabla \Phi$ satisfies the maximum principle for harmonic functions which states that a non-constant function always attains its maximum at the boundary of the domain \cite[Theorem 14.1]{Gilbarg2001}. Nevertheless, in average over all non-target regions, with a value of \SI{0.01477}{A m^{-2}} for M2E, the L2R (\SI{0.00345}{A m^{-2}}) and L1R (\SI{0.00249}{A m^{-2}}) optimized current flow fields show a factor of about 4.3 and 5.9 times lower current densities, resp. (Table \ref{Quantification}, third column). Overall, L2R and L1R thus have a much higher focality, which can also easily be seen in Figures \ref{FigDeep}-B1,-C1 and -D1. However, if no multi-channel tDCS device is available, the M2E approach provides an optimized bipolar electrode arrangement for a deep and tangential target vector. \newline
\rev{The deep target region does not seem to be located in a very deep region of the brain. It gets obvious that the deeper the target vector is located, the higher the averaged maximal current density in the brain compartment, especially in more lateral brain regions. Due to the maximum principle it is thus not possible to target deep regions without stimulating more lateral brain areas. However, many important target regions, e.g. auditory, motor or visual cortex, are located rather laterally, so that it will be possible to target with significant field strength in many applications. Moreover, in many applications of brain stimulation it might also not matter, if non-target regions are also involved, because the experimental setup focuses on the target region, for example when examining the change in event-related potentials (ERP) in pre- and post- tCS stimulation ERP measurements.}

\section{Conclusion and Outlook}
A novel optimization approach for safe and well-targeted multi-channel transcranial direct current stimulation has been proposed. Existence of at least one minimizer 
%and residual and objective convergence results have 
has been proven for the proposed optimization methods. For discretization of the respective minimization problems the finite element method was employed and the existence of at least one minimizer to the discretized optimization problems have been shown. For numerical solution of the corresponding discretized problem we employed the alternating direction method of multipliers. A highly-realistic six-compartment head model with white matter anisotropy was generated and optimized current density distributions were calculated and evaluated for a mainly tangential and a mainly radial target vector at superficial locations, an extended target area and a deeper mainly tangential target vector. The numerical results revealed that, while all approaches fulfilled the patient safety constraint,  the optimized current flow fields show significantly higher focality and, with the exception of the L2R for the deep target, higher directional agreement to the target vector in comparison to standard bipolar electrode montages. The higher directional agreement is especially distinct for the radial target vector. In all test cases, because of a more widespread distribution of injected and extracted surface currents, the L2R optimization procedure $\mathbf{(\bar{P}^{\alpha,0}_{\epsilon})}$ led to relatively weak current densities in the brain compartment. The L1R optimized current density distribution along the target direction was in all test cases stronger than the L2R one and might thus be able to induce more significant stimulation effects. The stimulation will thus enhance cortical excitability especially in the target regions, while it will as good as possible prevent too strong excitability changes in non-target regions. \newline
We were able to demonstrate that the M2E approach provides optimized bipolar electrode montages as long as the target is  mainly tangentially oriented. For radial targets, the M2E approach was unsatisfactory, an optimal bipolar electrode configuration might then consist of a small electrode placed directly above the target region with a distant return electrode or a small electrode over the target encircled by a ring return electrode, as proposed in \cite{Dmochowski2011}. \newline
A further application for the optimization method is transcranial magnetic stimulation (TMS). TMS uses externally generated magnetic fields to induce electrical currents to the underlying
brain tissue \cite{Gomez2013}. Because there is no safety limit for the total currents applied to the stimulating coils but a safety-threshold for painful muscle twitching \cite{Gomez2013}, the 
constrained optimization problem for multi-coil TMS is given as 
\begin{eqnarray*}
  \mathbf{(P_{TMS})} \qquad 
&- \int_{\Omega_t}\! \langle \sigma \nabla \Phi, \boldsymbol e \rangle \, \mathrm{d} \boldsymbol x& \rightarrow \min \\ 
& \text{subject to } &
 \omega | \sigma \nabla \Phi | \leq E_M \\ 
&& \nabla \cdot \sigma \nabla \Phi  =  - \nabla \cdot \sigma \frac{\partial \boldsymbol A(\boldsymbol x,t)}{\partial t} \quad \quad in \; \Omega \\
&& \langle \sigma \nabla \Phi, \boldsymbol n \rangle = - \langle \sigma \frac{\partial \boldsymbol A(\boldsymbol x,t)}{\partial t}, \boldsymbol n \rangle \quad on \;  \Gamma \\
 && \Phi = 0 \qquad \qquad \qquad \qquad \qquad \quad on \;\ \Gamma_D
\end{eqnarray*}
with $\boldsymbol A(\boldsymbol x,t)$ being the time-dependent magnetic vector potential and $E_M$ = \SI{450}{V m^{-1}} being the threshold for painful muscle twitching \cite{Gomez2013}. By designing the changes in the magnetic vector potential one can consider $\frac{\partial \boldsymbol A(\boldsymbol x,t)}{\partial t}$ as the optimization variables, respectively some parameters on which it depends linearly. The existence of at least one minimizer to the constrained optimization problem for TMS directly follows with similar arguments as in Theorem \ref{minQE}.\newline
Because the optimization method can be applied for both brain stimulation modalities, a combined tDCS and TMS optimization might outperform single modality tDCS or TMS optimizations, similar to what was shown for electro- (EEG) and magnetoencephalography (MEG) \cite{Aydin2014, Aydin2015}. While tCS is able to stimulate a radially oriented target, TMS is mainly not (like MEG is hardly able to detect radial sources \cite{Aydin2014, Aydin2015}). Possible applications of combined tDCS and TMS multi-channel and multi-coil optimization might thus be an improved stimulation of target regions containing both radial and tangential orientations or of deeper target regions. In order to induce  action potentials in deeper target regions, the induced current density should exceed the threshold for neuronal depolarization of \SI{150}{V m^{-1}} \cite{Gomez2013}. On the other hand, the threshold for painful muscle twitching of \SI{450}{V m^{-1}} must be kept \cite{Gomez2013}. The combination of the optimized tDCS and TMS current density fields might lead to higher current densities in the target and simultaneously reduced current density amplitudes in non-target regions. \newline
\rev{While a thorough mathematical analysis of our novel multi-array tDCS optimization method was derived and results for different target regions were presented, besides the first promising results presented in \cite{Homoelle2016}, we did not yet further compare our method to the existing approaches in the literature such as, e.g., \cite{Dmochowski2011, Sadleir2012, Im2008, Ruffini2014}. Such a comparison is one of our future research goals.}

\vspace{0.5cm}
\noindent
\textbf{Acknowledgements}
\noindent 
SW and CHW were supported by the priority program SPP1665 of the German Research Foundation, project WO1425/5-1. MB acknowledges support by ERC via Grant EU FP 7 - ERC Consolidator Grant 615216 LifeInverse and by the German Science Foundation DFG via EXC 1003 Cells in Motion Cluster of Excellence, M\"unster, Germany.

\end{document}